\documentclass[12pt]{article}

\usepackage{subcaption}
\usepackage{amsmath}
\usepackage{mathtools}
\usepackage{amssymb}
\usepackage{amsthm}
\usepackage{mathrsfs}
\usepackage{tikz-cd}
\newtheorem{theorem}{Theorem}[section]
\newtheorem{lemma}[theorem]{Lemma}
\newtheorem{proposition}[theorem]{Proposition}
\newtheorem{corollary}[theorem]{Corollary}
\newtheorem{definition}[theorem]{Definition}

\newtheorem{axiom}{Axiom}
\newtheorem*{axiom*}{Axiom}

\newtheorem*{observation*}{Observation}

\DeclareFontFamily{U}{min}{}
\DeclareFontShape{U}{min}{m}{n}{<-> udmj30}{}

\DeclareMathOperator{\Sub}{Sub}
\DeclareMathOperator{\dom}{dom}
\DeclareMathOperator{\im}{im}

\DeclareMathOperator{\st}{\mathbf{st}}

\usepackage{booktabs} 
\usepackage{multirow}

\begin{document}

\title{Nonstandard proof methods in toposes}
\author{Jos\'e Siqueira}
\date{\small Gonville \& Caius College, University of Cambridge, England}

\maketitle

\begin{abstract}
We determine sufficient structure for an elementary topos to emulate E. Nelson's Internal Set Theory in its internal language, and show that any topos satisfying the internal axiom of choice occurs as a universe of standard objects and maps. This development allows one to employ the proof methods of nonstandard analysis (transfer, standardisation, and idealisation) in new environments such as toposes of $G$-sets and Boolean \'etendues.
\end{abstract}



\section{Introduction}
\label{sec:introduction}
The axioms of Internal Set Theory (IST) concisely capture the proof methods employed in Robinson's nonstandard analysis, with the goal of making such tools readily available to the working mathematician without a background in model theory \cite{nelson1977}. The device known as the \emph{internal language} of a category can play a similar role: one need not know the intricacies of categorical logic in order to do internal reasoning (only learn to operate within the appropriate logic, e.g., intuitionistic logic), which can then be externalised to statements about unfamiliar objects if they so fancy. Here we combine both gadgets into one: we introduce structure a topos can have in order for its internal language to support the tools of nonstandard analysis (using IST as proxy). 

This is, of course, not the first attempt at a categorification of nonstandard methods \cite{TTFNSE, MOERDIJK199537, palmgren1998,  PALMGREN2001275}, but it comes from a new perspective which is rather natural and addresses all three schema from IST instead of focusing on just one in isolation.
The point of view is that the additional axiom schemes of internal set theory (relative to ZFC) are the result of relationships between \emph{hyperdoctrines} \cite{lawvere}, tools which allow us to abstract away from the ideas of \emph{internal formula}, \emph{internal formula with standard parameters}, and \emph{external formula} while preserving their first-order logical features. Starting from a set theory as a template also allows us to leverage the well-known connections between topos theory and set theory \cite{LawvereETCS,  MITCHELL1972261, OSIUS197479, SHULMAN2019465} --- taking the ``internal universe" to be an elementary topos with extra structure that satisfies certain properties ought to be seen as a straight-up generalisation of starting from a model of ZFC with an additional predicate that obeys some axioms.


It is possible to formulate the topos-theoretic analogue to the axioms of IST in abstract and without reference to a pre-conceived notion of ``standard element" \cite{TriposModels}, but this paper is concerned with toposes whose behaviour mirror that of the set-theoretic models of Nelson's IST: those for which \emph{there is} a well-behaved `class' of nonstandard elements, the meaning of which will be made clear in Section~\ref{sec:NelsonToposes}. Reasoning resorting to the ``nonstandard principles" of transfer, standardisation and idealisation can be carried out internal to such a topos to provide valid proofs, as long as it is otherwise intuitionistic. This enables the formulation of arguments such as those presented in  \cite{nelson1977}, but that still hold in other settings. 

The paper is structured in three sections. In Section 2, we introduce classical Internal Set Theory, which will be the inspiration for the theory developed in the paper, and settle some terminology. We will assume familiarity with the notions of (first-order) \emph{hyperdoctrine} and of \emph{tripos}; readers not acquainted with them are encouraged to read \cite{lawvere}, \cite{hyland_johnstone_pitts_1980}, and \cite{van2008realizability}.

Section 3 concerns the implementation of a topos-theoretic version of Internal Set Theory. We show that the category of standard objects and standard maps of a model of Internal Set Theory is a topos, which embeds logically and conservatively into the corresponding category of IST-sets. This serves as inspiration for the new notion of Nelson structure, which captures the idea of a topos admitting a nonstandard extension that includes a `predicate of standard elements' for each of its objects. We then discuss what it means for such a structure to satisfy topos-theoretic versions of transfer, standardisation, and idealisation, which validate these methods within the internal logic. 

Section 4 is devoted to building a class of examples of \emph{Nelson models} --- those Nelson structures that satisfy the topos-theoretic version of IST axioms. This is done via exploiting the notion of \emph{internal ultrafilter} \cite{VOLGER1975345} to implement Robinson-style ultrapower models. We show that any topos satisfying the internal axiom of choice serves as the universe of standard objects and morphisms of a Nelson structure that satisfies transfer and standardisation and, moreover, that there is a Nelson structure for which a limited form of idealisation also holds. This makes it possible to employ the methods of nonstandard analysis in other settings (e.g., toposes of $G$-sets or a Boolean \'etendue).

\section{Background}

We summarise here the specifics of Nelson's Internal Set Theory (IST), the categorical logic facts we will employ, and some associated notation. \\

The signature for IST consists of one relation symbol $\in$ of arity $2$, one relation symbol $\st$ of arity $1$, and no function symbols. A formula in this language is \emph{internal} if it is also a formula of ZFC, i.e. it does not involve the new predicate $\st$. The axioms of IST are those of Zermelo-Fraenkel set theory (with Choice), together with three additional axiom schema that determine the behaviour of sets satisfying $\st$, the \emph{standard sets}. For the record, they are\footnote{Here we use the abbreviations $\forall^{\st} x.\ \phi(x)$ and $\exists^{\st} x.\ \phi(x)$ to mean $\forall x.\ (\st(x) \Rightarrow \phi(x))$ and $\exists x.\ (\st(x) \land \phi(x))$, respectively.}:

\begin{axiom}[Transfer]
If $\phi(\Vec{x}, \Vec{t})$ is an internal formula with standard parameters $\Vec{t}\coloneqq t_1, \cdots, t_m$ that holds for all standard sets $\Vec{x}\coloneqq x_1, \cdots, x_n$, then it holds for all sets $\Vec{x}$:
\begin{equation*}
    (\forall^{\st} \Vec{x}.\ \phi(\Vec{x}, \Vec{t}) )\Rightarrow (\forall \Vec{x}.\ \phi(\Vec{x}, \Vec{t}))
\end{equation*}
\label{ax:transfer}
\end{axiom}

This means that the truth of \emph{internal formulae with standard parameters} needs to be verified for standard sets only. This is the main tool of nonstandard analysis in its classical formulation, which is often employed in the form of an easy consequence of its contrapositive:

\begin{proposition}[Uniqueness principle]
    Let $\phi(x,\Vec{t})$ be an internal formula with standard parameters. If there exists an unique set $x$ such that $\phi(x,\Vec{t})$, then such a set $x$ is standard.
\end{proposition}

\begin{axiom}[Standardisation]
Given any formula $\phi$ in the language of IST and a standard set $A$, there is an unique standard subset $A^{\phi}$ of $A$ whose standard elements are precisely those standard elements of $A$ that satisfy $\phi$.
\label{ax:standardisation}
\end{axiom}

This axiom is necessary since those of ZFC can only be used to infer facts about internal formulae (they have nothing to say about the alien formulae that make use of the new predicate $\st$) --- for example, given a set $A$, the class $\{x \in A \colon \st(x)\}$ is not necessarily a set, but as long as $A$ is standard we have a legitimate approximation that is close enough, namely the standard subset $A^{\st}$ of $A$. Note that the axiom only stipulates what the \emph{standard} elements of $A^{\phi}$ must be: nothing is assumed about its nonstandard elements.

\begin{axiom}[Idealisation]
Given an internal formula $R(x,y)$ (that possibly includes nonstandard parameters $\Vec{t}$), the following two statements are equivalent:
\begin{enumerate}
    \item  For every finite standard set $B$ there is a set $y_B$ such that $R(x,y_B)$ for every element $x$ of $B$;
    \item There is a set $y$ such that $R(x,y)$ for every standard set $x$.
\end{enumerate}
\label{ax:idealisation}
\end{axiom}

This is Nelson's counterpart to saturation properties in Robinson's approach. This is used in nonstandard analysis to show the existence of infinitesimal real numbers (that is, nonzero real numbers smaller than any standard one) and other important facts. We single out just one for now, obtained from considering the case $R(x,y)$ is the formula $\neg (x=y)$: the standard sets that contain only standard elements are precisely the finite ones.

The resulting set theory is a conservative extension of ZFC \cite{nelson1977}, therefore any internal theorem of IST is provable in ZFC and we may still make use of ZFC theorems within IST. Combining them with the new principles, it is possible to write down legitimate proofs in the style of the old geometers (and young physicists) --- examples can be found in \cite{nelson1977}.\\


We finish this section with a few remarks on categorical logic.\\

By a \emph{hyperdoctrine} $P$ over a category $\mathcal{C}$ with finite limits (its \emph{category of types}), we mean a $\mathcal{C}$-indexed Heyting pre-algebra such that:
\begin{itemize}
    \item  For each $f\colon A \to B$ in $\mathcal{C}$, the functor $P(B) \xrightarrow{P(f)} P(A)$ of preorders preserves implication and has both a left adjoint $\exists(f)$ and a right adjoint $\forall(f)$;
    \item Such adjoints satisfy the Beck-Chevalley condition: given a pullback square $\begin{tikzcd}
X \arrow[dr, phantom, "\lrcorner", very near start]
 \arrow[]{r}{k} \arrow[swap]{d}{h}& A \arrow[]{d}{f} \\
B \arrow[swap]{r}{g}& Y
\end{tikzcd}$ in $\mathcal{C}$, there are isomorphisms 
\begin{equation*}
  {P(f) \circ \exists(g) \cong \exists(k) \circ P(h)} \ \text{ and }\  {P(f) \circ \forall(g) \cong \forall(k) \circ P(h)}.  
\end{equation*}
\end{itemize}
By a \emph{Heyting transformation} between hyperdoctrines $P$ and $Q$ over $\mathcal{C}$, we mean a $\mathcal{C}$-indexed functor $\alpha \colon P \to Q$ that is a morphism of Heyting pre-algebras componentwise, and that commutes with quantification.

A hyperdoctrine is a \emph{tripos} if it admits a generic predicate: there is an object $\Sigma \in \mathcal{C}$ and predicate $\sigma \in P(\Sigma)$ such that, for all $\phi \in P(A)$ there is a map $\chi_{\phi}\colon A \to \Sigma$ in $\mathcal{C}$ with $\phi \cong P(\chi_{\phi})(\sigma)$. Although we won't resort to it directly, it is worth noting that every tripos induces a topos via the \emph{tripos-to-topos construction} (see \cite{hyland_johnstone_pitts_1980} or \cite{van2008realizability} for details). Most importantly for us, one can soundly interpret intuitionistic first-order logic within a hyperdoctrine, and intuitionistic higher-order logic in any tripos \cite{van2008realizability}.

As a final remark, note that if $F \colon \mathcal{C} \to \mathcal{D}$ preserves pullbacks and $P$ is a hyperdoctrine over $\mathcal{D}$, then the $\mathcal{C}$-indexed category $F^*(P)$ given by $F^*(P)(A)\coloneqq P(F(A))$ and $F^*(P)(f)\coloneqq P(F(f))$ for each object $A$ and morphism $f\colon A \to B$ in $\mathcal{C}$ is a hyperdoctrine over $\mathcal{C}$. We refer to it as the \emph{reindexing} of $P$ along $F$.\\

\section{Internal Set Theory in a topos}
\label{sec:NelsonToposes}
This section introduces sufficient conditions for an indexed family of predicates to behave as a `class of standard elements', and provides faithful translations of the original Nelson axioms to such contexts. We then show that internal set theory can be carried out internal to a topos satisfying the topos-theoretic version of Nelson's axioms.

The first step is to nail down the notions of standard, internal, and external predicates for a topos, for which it is helpful to make a categorical analysis of the classical (set-theoretic) models of IST. We fix a model $\mathbb{U}$ of Internal Set Theory in what follows.

\begin{proposition}
A finite cartesian product of standard sets in $\mathbb{U}$ is standard, and the associated projection functions $\pi_j \colon \prod_{i=1}^n A_i \to A_j$ are standard maps.
\label{prop:StProducts}
\end{proposition}

\begin{proof}
Note $0\coloneqq \emptyset$ is standard by the Uniqueness Principle, as Extensionality implies it is the only set satisfying the internal formula $\forall x\ \neg(x \in z)$ without parameters. Similarly, $1=\{0\}$ is uniquely defined by an internal formula with the standard set $0$ as its only parameter, hence is standard --- this means an empty product is standard\footnote{More generally, every finite von Neumann ordinal is standard.}.

For simplicity, we will write the argument for binary products. Let $A$ and $B$ be standard sets. As $A \times B$ is the unique set satisfying the internal formula
\begin{equation*}
    \mathbf{prod}_{A,B}(z)= x \in z \iff (\exists a\ \exists b\ (a \in A \land b \in B \land z =\langle a, b \rangle))
\end{equation*}
with standard parameters $A$ and $B$, it has to be standard (here $\langle a, b \rangle$ denotes the Kuratowski pair $ \{ \{a\}, \{a,b\}\}$). The projection $\pi_A\colon A \times B \to A$ is then standard, as it is the only set satisfying the internal formula
\begin{equation*}
    \mathbf{fun}_{A\times B, A}(z) \land \forall a,a' \in A\ \forall b \in B\ \langle \langle a,b\rangle, a'\rangle \in z \Leftrightarrow a=a')
\end{equation*}
with free variable $z$ and standard parameters $A$ and $B$, where $\mathbf{fun}_{A\times B, A}(z)$ is the internal formula that formalises the statement that the set $z$ is a function $A \times B \to A$.
\end{proof}

\begin{corollary}
There is a category $\mathbf{Set}^{\st}_{\mathbb{U}}$ whose objects are the standard sets of $\mathbb{U}$ and whose morphisms are the standard functions between them.
\end{corollary}

\begin{proof}
The sets and functions of $\mathbb{U}$ form a category by virtue of $\mathbb{U}$ being, in particular, a model of ZFC. Moreover, it is immediate from Proposition~\ref{prop:StProducts} that the identity map on a standard set $A$ is standard. Together with the Uniqueness Principle, we can show that standard functions compose: if $f\colon A \to B$ and $g\colon B \to C$ are maps in $\mathbf{Set}^{\st}_{\mathbb{U}}$, then the composite $g\circ f$ is uniquely defined by the internal formula
\begin{equation*}
    \forall x\ \left( x \in z \Leftrightarrow \exists (a,c) \in A \times C\  (x=\langle a,c\rangle \land \exists b \in B\ [\langle b,c\rangle \in g \land \langle a,b\rangle \in f])\right).
\end{equation*} 
with free variable $z$ and standard parameters $A \times C, B, f, g$, and the projection maps, thus standard. 
\end{proof}

\begin{lemma}
The category $\mathbf{Set}^{\st}_{\mathbb{U}}$ has all finite limits and the inclusion functor $\iota\colon \mathbf{Set}^{\st}_{\mathbb{U}} \hookrightarrow \mathbf{Set}_{\mathbb{U}}$ creates them.
\end{lemma}

\begin{proof}
It is enough to show that it has finite products and equalisers, and that $\iota$ creates them.

Given $f,g\colon A \to B$ in $\mathbf{Set}^{\st}_{\mathbb{U}}$, we note that the formula 
\begin{equation*}
    E_{f,g}(x) \coloneqq x \in A \land (f(x)=g(x))
\end{equation*}
is internal with standard parameters $A, f, g$, and that the inclusion of the subset of $A$ obtained from it by Separation is an equalizer $E \xrightarrow{e} A$ of $f$ and $g$ in $\mathbf{Set}_{\mathbb{U}}$. By Extensionality it is the unique set satisfying the internal formula $\psi(z)= \forall x\ (x \in z \Leftrightarrow E_{f,g}(x))$ with standard parameters $A, f, g$, hence is standard.

Similarly, the inclusion map $E \xrightarrow{e} A$ is standard. Suppose $E' \xrightarrow{e'} A$ is standard, and $f \circ e' = g \circ e'$. By the universal property of the equaliser in $\mathbf{Set}_{\mathbb{U}}$, there is a unique function $\gamma \colon E' \to E$ with $e \circ \gamma=e'$. But then  $\gamma$ is the only set satisfying the internal formula
\begin{equation*}
    \textbf{fun}_{E',E}(z) \land (\forall \langle x, y\rangle \in E'\times E)[\langle x, y\rangle \in z \Leftrightarrow \forall a \in A ( \langle y, z\rangle \in e \Leftrightarrow \langle x, z\rangle\in e')]
\end{equation*}
with free variable $z$ and standard parameters $E'\times E, A, e, e'$. By the Uniqueness Principle $\gamma$ is standard, thus $E \xrightarrow{e} A$ is an equaliser of $f$ and $g$ in $\mathbf{Set}^{\st}_{\mathbb{U}}$.

For finite products, we augment Proposition~\ref{prop:StProducts} by showing that the maps induced by the universal property in $\mathbf{Set}_{\mathbb{U}}$ of the cartesian product of standard sets is standard, which makes such a cartesian product a categorical product in $\mathbf{Set}^{\st}_{\mathbb{U}}$ (as the inclusion $\iota\colon \mathbf{Set}^{\st}_{\mathbb{U}} \hookrightarrow \mathbf{Set}_{\mathbb{U}}$ is obviously faithful).

First note that if $A$ is a standard set, then uniqueness of the function $A \to 1$ makes it standard, hence $1=\{0\}$ is a terminal object in $\mathbf{Set}^{\st}_{\mathbb{U}}$. Now let $A$, $B$ and $C$ be standard sets, and suppose there are standard maps $C \to A$ and $C \to B$. Then the unique function $C \to A \times B$ that exists by force of the universal property of $A \times B$ in $\mathbf{Set}_{\mathbb{U}}$ has to be standard, as the only function satisfying some internal formula with standard parameters. 

We conclude $\mathbf{Set}^{\st}_{\mathbb{U}}$ has finite products and equalisers, hence finite limits. Moreover, by construction, the inclusion $\iota \colon \mathbf{Set}^{\st}_{\mathbb{U}} \hookrightarrow \mathbf{Set}_{\mathbb{U}}$ preserves (and in fact creates) them.\\
\end{proof}

We remark that the uniqueness of the map induced by the universal property of the considered limit is essential for the argument to carry through --- it is not clear that, say, a weak product of standard objects has to be standard.\\

Similar ideas provide the next result, which will serve as inspiration for the axiomatics that follows:

\begin{theorem}
 The category $\mathbf{Set}^{\st}_{\mathbb{U}}$ is a topos and the inclusion functor $\iota\colon \mathbf{Set}^{\st}_{\mathbb{U}} \hookrightarrow \mathbf{Set}_{\mathbb{U}}$ is conservative and logical.
 \label{thm:StandardSetsLogicalEmbedding}
\end{theorem}
\begin{proof}

We show that $\mathbf{Set}^{\st}_{\mathbb{U}}$ has power objects, which are built as in $\mathbf{Set}_{\mathbb{U}}$.

This is the case once more by the Uniqueness Principle, as for any standard set $A$ the power set $\mathcal{P}(A)$ is the unique set satisfying the internal formula $\forall x (x \in z \iff x \subseteq A)$ with standard parameter $A$, hence is standard. Similarly, $\in_A \coloneqq \{\langle a, S\rangle \in A \times \mathcal{P}(A) \colon a \in S\}$ and its inclusion map into $A \times \mathcal{P}(A)$ are standard, as the membership predicate is internal and we now know that $A \times \mathcal{P}(A)$ is standard. 

If $R \rightarrowtail A \times B$ is a monomorphism in $\mathbf{Set}^{\st}_{\mathbb{U}}$, then it is also a monomorphism in $\mathbf{Set}_{\mathbb{U}}$ by virtue of $\iota$ preserving pullbacks. By the universal property of the power set in $\mathbf{Set}_{\mathbb{U}}$, there is a unique function $\chi_R\colon B \to \mathcal{P}(A)$ such that 
\begin{equation}
    \begin{tikzcd}
      R \arrow[rightarrowtail, swap]{d}{r} \arrow[]{r}{\lambda}& \in_A \arrow[rightarrowtail]{d}{}\\
      A\times B \arrow[swap]{r}{1_A \times \chi_R} & A \times \mathcal{P}(A)
    \end{tikzcd}
    \label{eq:4.1}
\end{equation}
is a pullback in $\mathbf{Set}_{\mathbb{U}}$. This means $\chi_R$ is the only set satisfying the internal formula
\begin{align*}
    \mathbf{fun}_{B,\mathcal{P}(A)}(z) \land [(1_A \times \chi_R) \circ r=\in_A \circ \lambda] \land \forall L,l,h (\textbf{mono}(l)_{L, A \times B} \land \mathbf{fun}_{L,\in_A}(h) \\
    \land [(1_A \times \chi_R) \circ l = \in_A \circ h] \Rightarrow \exists!\gamma(\mathbf{fun}_{L,R}(\gamma) \land [r\circ \gamma=l] \land [\lambda \circ \gamma=h]),
\end{align*}
 where $\mathbf{mono}_{L, A\times B}$ is the internal formula stating that $l$ is an injective map $L \to A \times B$. This formula has standard parameters, hence $\chi_R$ is standard.
 
 As the inclusion $\iota \colon \mathbf{Set}^{\st}_{\mathbb{U}} \hookrightarrow \mathbf{Set}$ creates pullbacks, Diagram~\ref{eq:4.1} is a pullback in $\mathbf{Set}^{\st}_{\mathbb{U}}$, exhibiting $(\mathcal{P}(A), \in_A)$ as a power object of $A$ and establishing that $\mathbf{Set}^{\st}_{\mathbb{U}}$ is a topos and the functor $\iota$ is logical (by construction). 
 
 To see that it is also conservative, suppose there is a bijection between standard sets $A$ and $B$. Then the internal formula
 \begin{equation*}
     \mathbf{fun}_{A,B}(f)\land [\forall\ x \in A\ \exists !\ y \in B\ (x,y) \in f]
 \end{equation*}
 with free variable $f$ and standard parameters $A$ and $B$ is satisfiable. By the existential form of Transfer, there is standard $f$ satisfying this formula, which is then an isomorphism in $\mathbf{Set}^{\st}_{\mathbb{U}}$.
\end{proof}
\begin{observation*}
The proof of Theorem~\ref{thm:StandardSetsLogicalEmbedding} did not use Idealisation nor Standardisation, only Transfer and ZFC theorems. However, full Transfer (and not just $\Delta_0$-Transfer, the bounded form of it to be employed later) was needed in order to establish preservation of power objects. Moreover, the functor $\iota$ does not exhibit $\mathbf{Set}^{\st}_{\mathbb{U}}$ as a subtopos of $\mathbf{Set}_{\mathbb{U}}$, else it would be an equivalence (as a logical direct image of a geometric morphism) --- but $\mathbf{Set}^{\st}_{\mathbb{U}}$ doesn't necessarily have all $\mathbf{Set}_{\mathbb{U}}$-indexed colimits (e.g., a coproduct indexed by a nonstandard set).
\end{observation*}

Our goal is to abstract away from this scenario by stipulating the properties categories should have in order to play the roles of `standard universe' and `internal universe' (classically $\mathbf{Set}^{\st}_{\mathbb{U}}$ and $\mathbf{Set}_{\mathbb{U}}$, respectively).\\

\subsection{Nelson structures}

We start by noting that the fact that the topos $\mathbf{Set}^{\st}_{\mathbb{U}}$ faithfully embeds into the topos of sets $\mathbf{Set}_{\mathbb{U}}$ via the logical functor $\iota$ is rather strong --- we generally only need first-order logic to be preserved, so the functor between the would-be standard universe and the internal universe need not be more than a faithful Heyting functor\footnote{In the examples to be considered in Section 3, however, the analogous inclusion will still be logical.}. In this light, there is no reason to request the standard universe to be a topos, but simply a Heyting category  (although we will assume the internal one to be so, as to exploit the richer internal logic).

We have not yet accounted for Standardisation and Idealisation. To do so (and to give a translation of Transfer that closely follows Nelson), we need extra structure that gives us access to a notion of `external formula' and of `standard element':

\begin{definition}
A Nelson structure is comprised of:
\begin{itemize}
    \item A topos $\mathcal{I}$ of internal objects and maps;
    \item A Heyting category $\mathbb{S}$ of standard objects and standard maps;
    \item A faithful Heyting functor $\iota \colon \mathbb{S} \hookrightarrow \mathcal{I}$;
    \item A hyperdoctrine $\mathfrak{X}$ over $\mathcal{I}$ of external predicates;
    \item A Heyting transformation of hyperdoctrines $[\cdot]\colon \Sub_{\mathcal{I}} \to \mathfrak{X}$;
    \item A family of predicates $\sigma_I \in \mathfrak{X}(I)$ for each object of $\mathcal{I}$, which restricts to a lax natural transformation $\sigma \colon 1 \to \mathfrak{X}^*\coloneqq \iota^*(\mathfrak{X})$ satisfying $\sigma_1\cong\top$. 
\end{itemize}  
\label{def:NelsonStructure}
\end{definition}

We picture the topos $\mathcal{I}$ as the ``internal universe", within which lies a ``standard universe" $\mathbb{S}$ as a subcategory. The hyperdoctrine $\mathfrak{X}$ is a \emph{nonstandard extension} that collects the ``external predicates'' about internal objects, with the transformation $[\cdot]\colon \Sub_{\mathcal{I}} \to \mathfrak{X}$ enabling us to see an internal property externally. It includes a special external predicate $\sigma_I$ for each internal object $I$, which plays the role of the external formula $\st(x)$.

As usual, the subobjects of some $I \in \mathcal{I}$ ought to be seen as predicates over $I$, which we will call \emph{internal predicates}. In the motivating example where $\mathcal{I}=\mathbf{Set}_{\mathbb{U}}$, such subobjects are in correspondence with the internal formulae of IST via the Axiom of Separation. Similarly, the subobjects of $A \in \mathbf{Set}^{\st}_{\mathbb{U}}$ correspond to the internal formulae \emph{with standard parameters} and are thus named \emph{standard predicates}: this is because the subset of a standard set $A$ obtained from such a formula by Separation has to be standard by the Uniqueness Principle and the Axiom of Extensionality and, conversely, if $S$ is a standard subset of $A$, then $x \in S$ is an internal formula with standard parameters. 

Thinking about the subobjects of standard objects as a replacement for internal formulae with standard parameters, we determine that $\mathbb{S}$ ought to be a Heyting category (if not necessarily a topos) --- note that internal formulae with standard parameters are closed under first-order logical connectives, justifying that assumption in Definition~\ref{def:NelsonStructure}.

Formulating Standardisation requires the existence of the hyperdoctrine $\mathfrak{X}$ of external predicates, which not only includes the internal ones via the Heyting transformation $[\cdot]\colon \Sub_{\mathcal{I}} \to \mathfrak{X}$, but also the standard predicates about standard objects through composition with the faithful Heyting functor $\iota$ applied at the level of subobjects.  Nelson's formulation of the Axiom of Standardisation only allows one to take the standardisation of an external formula relative to a standard set, so in an effort to replicate his formalism it is natural to reindex the hyperdoctrine $\mathfrak{X}$ over standard objects and get a new hyperdoctrine $\mathfrak{X}^*$ over $\mathbb{S}$.

The requirement for the predicates $\sigma_I$ to induce a lax transformation $\sigma\colon 1 \to \mathfrak{X}^*$ translates to the image of a standard element under a standard map (i.e., a morphism in the subcategory $\mathbb{S}$) being once again standard. The last condition that $\sigma_1 \cong \top$ in $\mathfrak{X}^*(1)$ (which is a mild consequence of Idealisation in classical IST) is a smallness restriction that stipulates that every element of a terminal object is standard; it will be used to show that our topos-theoretic versions of standardisation and transfer provide the usual ones internally.\\

As $\mathfrak{X}$ is a hyperdoctrine, it comes with its own quantifiers, which we see as ``external''.  We can define new quantifiers from them by employing the standardness predicates $\sigma_I$: for each internal map $I \xrightarrow{v} J$ in $\mathcal{I}$ and $\psi \in \mathfrak{X}(I)$ we have predicates $\forall^{\sigma}_v(\psi)\coloneqq \forall^{\mathfrak{X}}_v (\sigma_I \Rightarrow \psi)$ and $\exists^{\sigma}_v(\psi)\coloneqq \exists^{\mathfrak{X}}_v (\sigma_I\land \phi)$ in $\mathfrak{X}(J)$.\\

The maps $\psi \mapsto \forall^{\sigma}_v(\psi), \exists^{\sigma}_v(\psi)$ do not define functors as they are, but they do so when we replace the ordering on predicates with the \emph{standard elements order}: for $\psi_1,\psi_2 \in \mathfrak{X}(I)$, we write $\psi_1 \leq^{\st} \psi_2$ if $\psi_1 \land \sigma_I \leq \psi_2$, and write $\psi_1 \cong^{\st} \psi_2$ if $\psi_1 \leq^{\st} \psi_2$ and $\psi_1 \leq^{\st} \psi_2$.  This order is thus named since in the classical set-theoretic context $\phi \leq^{\st} \psi$ means that whenever $x_1, \cdots, x_n$ are standard and $\phi(x_1,\cdots, x_n)$ holds, then $\psi(x_1, \cdots, x_n)$ holds.\\

Clearly, $\psi_1 \leq \psi_2$ implies $\psi_1 \leq^{\st}\psi_2$, and logical connectives preserve the standard elements order:

\begin{proposition}
Let $a,b,c \in \mathfrak{X}^*(A)$. If $b \leq^{\st}c$, then 
\begin{enumerate}
    \item $(a \Rightarrow b) \leq^{\st} (a \Rightarrow c)$;
    \item $(a \lor b) \leq^{\st} (a \lor c)$.
    \item $(a \land b) \leq^{\st} (a \land c)$.
\end{enumerate}
\label{prop:ImplicationPreservesLeqSt}
\end{proposition}
\begin{proof}
Omitted.
\end{proof}
This is also true of the new quantifiers:
\begin{proposition}
Let $v \colon I \to J$ be an internal map. Then  $\forall^{\sigma}_v \colon \mathfrak{X}(I) \to \mathfrak{X}(J)$ and $\exists^{\sigma}_v\colon \mathfrak{X}(I) \to \mathfrak{X}(J)$ preserve the standard elements order.
\end{proposition}
\begin{proof}
    If $\psi_1 \leq^{\st} \psi_2$ in $\mathfrak{X}(I)$ then $\psi_1 \land \sigma_I \leq \psi_2 \land \sigma_I$, so applying quantification in $\mathfrak{X}$ yields $\exists^{\mathfrak{X}}_v(\psi_1 \land \sigma_I) \leq \exists^{\mathfrak{X}}_v(\psi_2 \land \sigma_I)$. Thus 
\begin{equation*}
    \exists^{\sigma}_v(\psi_1) \land \sigma_J \leq \exists^{\sigma}_v(\psi_1) \leq \exists^{\sigma}_v(\psi_2). 
\end{equation*}
This establishes that $\exists^{\sigma}_v$ preserves the standard elements order.

For universal quantification, note that we have $\forall^{\sigma}_v(\psi_1) \land \sigma_J \leq \forall^{\sigma}_v(\psi_2)$ if, and only if, $\mathfrak{X}(v) \forall^{\mathfrak{X}}_v(\sigma_I \Rightarrow \psi_1) \land \mathfrak{X}(v)(\sigma_J) \land \sigma_I \leq \psi_2$ by taking adjoints. But $\mathfrak{X}(v)\forall^{\mathfrak{X}}_v(\sigma_I \Rightarrow \psi_1) \leq (\sigma_I \Rightarrow \psi_1)$ via the counit of the adjunction defining $\forall^{\mathfrak{X}}_v$. It follows that
 \begin{equation*}
     \mathfrak{X}(v) \forall^{\mathfrak{X}}_v(\sigma_I \Rightarrow \psi_1) \land \mathfrak{X}(v)(\sigma_J) \land \sigma_I\leq (\sigma_I \Rightarrow \psi_1) \land \sigma_I \land \mathfrak{X}(v)(\sigma_J) \leq \psi_1 \land \sigma_I.
 \end{equation*}
But $\psi_1 \leq^{\st} \psi_2$, so the desired inequality holds and $\forall_v^{\sigma}$ preserves the standard elements order. 
\end{proof}

Substitution along a \emph{standard} map between standard sets also preserves the standard elements order: given $\psi_1,\psi_2 \in \mathfrak{X}^*(B)$ with $\psi_1\leq^{\st}\psi_2$ and standard $f \colon A \to B$, we have that 
\begin{equation*}
    \sigma_A \land \mathfrak{X}^*(f)(\psi_1) \leq   \mathfrak{X}^*(f)(\sigma_B) \land \mathfrak{X}^*(f)(\psi_1) = \mathfrak{X}^*(f)(\sigma_B \land \psi_1) \leq \mathfrak{X}^*(f)(\psi_2)
\end{equation*}
by lax naturality of $\sigma$ with respect to $f$, i.e., $\mathfrak{X}^*(f)(\psi_1) \leq^{\st} \mathfrak{X}^*(f)(\psi_2)$. \\

In fact, if we work with standard maps only then the lax naturality condition on $\sigma$ can be used to obtain an adjoint triple: 

\begin{proposition}
Given $f\colon A \to B$ in $\mathbb{S}$ and using the standard elements ordering on predicates\footnote{We use $\dashv^{\st}$ to emphasise that this adjunction is between categories pre-ordered by the standard elements order.}, we have an adjoint triple $\exists_{\iota(f)}^{\sigma} \dashv^{\st} \mathfrak{X}^*(f) \dashv^{\st} \forall_{\iota(f)}^{\sigma}$.
\label{prop:SigmaAdjunction}
\end{proposition}

\begin{proof}
To see that $\mathfrak{X}^{*}(f) \dashv^{\st} \forall^{\sigma}_{\iota(f)}$,  observe  that $\mathfrak{X}^{*}(f)(\phi) \leq^{\st} \psi$ precisely when $\phi  \leq \forall_f^{\mathfrak{X}^*}(\sigma_A \Rightarrow \psi) \eqqcolon \forall_{\iota(f)}^{\sigma}(\psi)$ by using the adjunction $\mathfrak{X}^*(f) \dashv \forall^{\mathfrak{X}}_{\iota(f)}$, whence ${\phi \leq^{\st} \forall^{\sigma}_{\iota(f)}(\psi)}$. 

Conversely, if $\phi \land \sigma_B \leq \forall_f^{\mathfrak{X}^*}(\sigma_A \Rightarrow \psi)$ then $\mathfrak{X}^{*}(f)(\phi \land \sigma_B) \leq (\sigma_A \Rightarrow \psi)$. But $\sigma_A \leq \mathfrak{X}^{*}(f)(\sigma_B)$ and $\mathfrak{X}^{*}(f)$ preserves meets, so $\mathfrak{X}^{*}(f)(\phi) \land \sigma_A \leq \psi$, i.e., $\mathfrak{X}^{*}(f)(\phi) \leq^{\st} \psi$.

Similarly, $\exists^{\sigma}_{\iota(f)} \dashv^{\st} \mathfrak{X}^*(f)$: if $\exists^{\sigma}_{\iota(f)}(\phi) \leq^{\st} \psi$, then $\exists^{\mathfrak{X}^*}_f(\phi \land \sigma_A) \land \sigma_B \leq \psi$, thus $\phi \land \sigma_A \leq \mathfrak{X}^*(f)(\sigma_B \Rightarrow \psi)$. As $\sigma_A \leq \mathfrak{X}^{*}(f)(\sigma_B)$ (by lax naturality of $\sigma$ on standard maps) and $\mathfrak{X}^{*}(f)$ preserves meets, we have $\phi \land \sigma_A \leq \mathfrak{X}^*(f)(\psi)$, i.e., $\phi \leq^{\st} \mathfrak{X}^*(f)(\psi)$. These steps are clearly reversible, yielding the adjunction.
\end{proof}

This settles the basic structure on top of which we may lay out topos-theoretic counterparts to the IST axioms. Compared to the set-theoretic version, there are two major caveats to bear in mind:
\begin{enumerate}
    \item While one can make sense of unbounded quantification in toposes \cite{awodey_butz_simpson_streicher_2007, SHULMAN2019465}, it is natural to adapt to the typed nature of their internal languages and $\Delta_0$-Transfer only, i.e., restrict instances of Axiom~\ref{ax:transfer} to when the formula $\phi$ only has quantifiers bound by standard sets (this will be made precise shortly). For the same reasons, we employ a bounded form of Idealisation. Note that Standardisation (even in Nelson) is already bound to standard sets.
    \item As the logic of a general topos is intuitionistic, we may not deduce the existential form of Transfer from the universal one. This is classically a consequence of the fact that internal formulae with standard parameters are closed under negation, which in our setting amounts to requiring $\mathbb{S}$ to be a Boolean category; we'll abstain from doing so unless needed.
\end{enumerate}

\begin{definition}
The internal $\Delta_0$-formulae of IST are defined recursively as follows:
\begin{enumerate}
    \item An internal formula without quantifiers is an internal $\Delta_0$-formula;
    \item If $\phi$ and $\psi$ are internal $\Delta_0$-formulae, then so is $\phi \bowtie \psi$, where $\bowtie$ is any of the first-order logical connectives;
    \item If $\phi$ is an internal $\Delta_0$-formula with a free variable $x$ and $A$ is a standard set, then $(\exists x \in A)\ \phi(x)$ and $(\forall x \in A)\ \phi(x)$ are internal $\Delta_0$-formulae.\footnote{Here $(\exists x \in A) \ \phi(x)$ and $(\forall x \in A)\ \phi(x)$ are shorthands for $(\exists x)\ (x \in A \land \phi(x))$ and $(\forall x)\ (x \in A \Rightarrow \phi(x))$ respectively, of course.}
\end{enumerate}
\end{definition}

Transfer as in Nelson holds only for the standard predicates, which we envision as the subobjects of $\mathbb{S}$. Those of course embed into $\mathfrak{X}$ through the Heyting transformation $i\colon \Sub_{\mathbb{S}} \to \mathfrak{X}^*$ given by ${i_A \coloneqq (S \rightarrowtail A) \mapsto [\iota(S)]} \in \mathfrak{X}^*(A)$. Transfer can then be stated as the following pair of inequalities holding for each standard $f \colon A \to B \in \mathbb{S}$:
 \begin{equation}
      \begin{tikzcd}
 i_B(\exists_f^{\mathbb{S}}(S)) \leq^{\st} \exists^{\sigma}_{\iota(f)}(i_A(S)) & & \forall^{\sigma}_{\iota(f)}(i_A(S)) \leq^{\st} i_B(\forall_f^{\mathbb{S}}(S)).
 \end{tikzcd}
 \label{eq:TransferForGeneralisedNelson}
 \end{equation}
 
 Indeed, in classical IST, $i_B(\exists_f^{\mathbb{S}}(S))$ is the internal $\Delta_0$-formula 
 \begin{equation*}
     (\exists x \in A)(y=f(x) \land x \in S),
 \end{equation*}
with one free variable $y$ --- if $y$ is standard, then this is existential quantification applied to an internal formula with standard parameters, thus $(\exists^{\st} x \in A)(y=f(x) \land x \in S)$ can be deduced from classical Transfer. But that is precisely saying that $\exists^{\sigma}_{\iota(f)}(i_A(S))$ holds. 
 
 The right-hand inequality in~(\ref{eq:TransferForGeneralisedNelson}) similarly expresses classical universal Transfer but, unlike in Nelson's IST, we have to explicitly assume both inequalities in the topos context since we cannot deduce the first one from the second unless Booleaness is assumed. We also note that the reverse of the inequalities in~(\ref{eq:TransferForGeneralisedNelson}) always holds for trivial reasons, so that Transfer can be formulated as the statement that $\exists^{\mathfrak{X}^*}_f$ and $\exists^{\sigma}_f$ ``coincide on standard elements of standard predicates'':
 
 \begin{axiom}[Topos-theoretic transfer]
 For each map $f \colon A \to B$ and subobject $S \rightarrowtail A$ in $\mathbb{S}$, we have \begin{equation*}
      \begin{tikzcd}
 i_B(\exists_f^{\mathbb{S}}(S)) \cong^{\st} \exists^{\sigma}_{\iota(f)}(i_A(S)) & & \forall^{\sigma}_{\iota(f)}(i_A(S)) \cong^{\st} i_B(\forall_f^{\mathbb{S}}(S))
 \end{tikzcd}
 \end{equation*}
 \end{axiom}
 
Note that classical Transfer holds for unbounded formulae, but the topos-theoretic form of the axiom is analogous to $\Delta_0$-Transfer; as mentioned, this is natural due to the typed nature of the internal language of toposes. \\

The translation of Standardisation is similarly straightforward:
  
 \begin{axiom}[Topos-theoretic standardisation]
 For each $A \in \mathbb{S}$ and ${\phi \in \mathfrak{X}^*(A)}$, there is a unique $\phi^{\sigma} \in \Sub_{\mathbb{S}}(A)$ such that  ${\diamond_A(\phi) \coloneqq i_A(\phi^{\sigma}) \cong^{\st} \phi}$  (up to isomorphism).
 \end{axiom}
 
It is clear that this recovers the original Axiom of Standardisation in classical IST. We will soon see that in a Nelson structure that satisfies topos-theoretic standardisation the adjoint triple $\exists_{\iota(f)}^{\sigma} \dashv^{\st} \mathfrak{X}^*(f) \dashv^{\st} \forall_{\iota(f)}^{\sigma}$ can be used to construct a new hyperdoctrine over $\mathbb{S}$.\\

Formalising topos-theoretic idealisation is slightly more complex. For similar reasons as for Transfer we will describe a bounded form of the classical axiom, but with one important caveat --- there are many notions of finiteness that coincide in $\mathbf{Set}$, but differ in other toposes. One can thus conceive of a different form of idealisation for each notion of finiteness in a topos, but here we shall employ Kuratowski-finiteness:

\begin{definition}[$K$-finite subobject \cite{Johnstone:592033}]
 Let $A$ be an object of a topos $\mathcal{E}$. The \emph{object of $K$-finite subobjects of $A$}, $K(A) \rightarrowtail \wp(A)$, is the free semilattice generated by $A$ (equivalently, the sub-join-semilattice of $\wp(A)$). We say $A$ is $K$-finite if the top element $\top_A\colon 1 \to \wp(A)$ factors through $K(A) \rightarrowtail \wp(A)$.
\label{def:KFinite}
\end{definition}

A major advantage of using $K$-finiteness is that it is internally representable in any topos and does not require the existence of a natural number object \cite{Johnstone:592033}. This means that, in particular, the language of $\mathfrak{X}$ in any Nelson structure includes a type $K(\iota A)$ of $K$-finite predicates for each standard object $A$, and so there's also a predicate of `standard, finite subobjects of $A$' (namely, $\sigma_{K(\iota A)}$). Exploiting that, we may formulate the following version of idealisation:

\begin{axiom}[$K$-finite idealisation]
Let $A, B \in \mathbb{S}$ be standard objects in a Nelson structure and $r \in \Sub_{\mathcal{I}}( \iota (A\times B))$ be an internal predicate. Then 
\begin{align*}
    \mathfrak{X} \models   \forall z\colon K(\iota A).\ \sigma_{K(\iota A)} \Rightarrow \exists y\colon \iota B.\ \forall x\colon \iota A .\left( x \in_{\iota A} z \rightarrow [r](x,y)\right)
\end{align*}
implies
\begin{align*}
        \mathfrak{X} \models\exists y\colon \iota B. \forall x\colon \iota A.\left(\sigma_{\iota A} \Rightarrow [r](x,y)\right),
\end{align*} where the interpretation used is the canonical one for the internal language of the hyperdoctrine $\mathfrak{X}^*$.
\label{ax:KFiniteIdealisation}
\end{axiom}

An instance of the axiom of $K$-finite idealisation as above is said to be $B$-\emph{bounded}.\\

We remark that the toposes to be considered in Section~\ref{sec:UltrapowerModels} are Boolean, and thus the usual notions of finiteness all coincide \cite{Johnstone:592033}. Moreover, logical functors preserve $K$-finiteness, so if the Nelson structure has a topos as its standard universe $\mathbb{S}$ and the functor $\iota \colon \mathbb{S} \to \mathcal{I}$ is logical (as is the case for classical IST and the structures of Section~\ref{sec:UltrapowerModels}), then Axiom~\ref{ax:KFiniteIdealisation} can be recast in terms of $\mathfrak{X}^*$ by using the isomorphism $\iota(K(A)) \cong K(\iota(A))$.

 \subsection{Nelson models}
 
We will now study the effects of topos-theoretic standardisation holding in a Nelson structure in isolation; no form of Transfer nor Idealisation will be assumed.\\
 
First observe that the standard elements order restricts to $\Sub_{\mathbb{S}}$: given $A \in \mathbb{S}$ and subobjects $S_1$ and $S_2$ of $A$, we set $S_1 \leq^{\st} S_2$ if $i_A(S_1) \leq^{\st} i_A(S_2)$.\footnote{It is clearly the case that $S_1 \subseteq S_2$ implies $S_1 \leq^{\st} S_2$, as expected.}. This may not be equivalent to the usual ordering of predicates in $\Sub_{\mathbb{S}}$ in the absence of Transfer, even in classical IST. We can however equip the underlying set of each $\Sub_{\mathbb{S}}(A)$ with the new order; this yields a Heyting pre-algebra denoted by $\Sub^{\st}(A)$, with finite meets, joins, and implications given by the standardisation of (the inclusion of) their counterparts in $\Sub_{\mathbb{S}}(A)$:

 \begin{lemma}
In a Nelson structure satisfying topos-theoretic standardisation, there is an $\mathbb{S}$-indexed functor $(-)^{\sigma} \colon \mathfrak{X}^* \to \Sub^{\st}$ whose components map finite meets to finite intersections, finite joins to finite unions, and implications to implications.
 \label{lem.StandardisationPreservesConnectives}
 \end{lemma}
 
 \begin{proof}
Recall that $i \colon \Sub_{\mathbb{S}} \to \mathfrak{X}^*$ is a Heyting transformation, and let $A$ be an object in $\mathbb{S}$. We start by proving that $(-)^{\sigma}$ maps finite meets to intersections. \\ 
We have that $i_A(A)=\top_A \cong^{\st} \top_A$, hence $A=(\top_A)^{\sigma}$ by uniqueness of standardisation. Similarly, given $\phi,\psi \in \mathfrak{X}^*(A)$, we have ${i_A(\phi^{\sigma} \cap \psi^{\sigma})} \cong i_A(\phi^{\sigma}) \land i_A(\psi^{\sigma}) \cong^{\st} \phi \land \psi$ by the defining property of standardisations and Proposition~\ref{prop:ImplicationPreservesLeqSt}, thus $(\phi\land\psi)^{\sigma} = \phi^{\sigma} \cap \psi^{\sigma}$ by uniqueness.  In much the same way 
\begin{equation*}
    i_A(\phi^{\sigma} \Rightarrow \psi^{\sigma}) \cong i_A(\phi^{\sigma}) \Rightarrow i_A(\psi^{\sigma}) \cong^{\st} (\phi \Rightarrow \psi)
\end{equation*}
by Proposition~\ref{prop:ImplicationPreservesLeqSt}, so $(-)^{\sigma}$ preserves implications. The same goes for finite joins. \\
The same idea proves naturality of $(-)^{\sigma}$. First note that if $f \colon A \to B \in \mathbb{S}$ is a standard map and $\phi \in \mathfrak{X}^*(B)$, then the pullback $f^{-1}(\phi^{\sigma})$ is a subobject of $A$ in $\mathbb{S}$. Now since $i_B(\phi^{\sigma}) \cong^{\st} \phi$, we have that $\mathfrak{X}^*(f)(i_B(\phi^{\sigma})) \cong^{\st} \mathfrak{X}^*(f)(\phi)$ as we know that substitution over a standard map preserves the standard elements order. But $i$ is a natural transformation, thus $i_A(f^{-1}(\phi^{\sigma})) \cong^{\st} \mathfrak{X}^*(\phi)$. We conclude that $f^{-1}(\phi^{\sigma}) = \mathfrak{X}^*(\phi)^{\sigma}$, as needed.
 \end{proof}
 
 Verifying that $\Sub^{\st}(A)$ is a Heyting pre-algebra for each $A \in \mathbb{S}$ in the advertised manner is now trivial. We shall see that these are the predicates for a new hyperdoctrine over $\mathbb{S}$.\\

 Given a standard map $f\colon A \to B \in \mathbb{S}$, the fact that $\mathfrak{X}^*(f)$ preserves the standard elements order provides a functor $\Sub^{\st}(B) \to \Sub^{\st}(A)$; we'll denote it by  $\Sub^{\st}(f)$. Indeed, if $S_1 \leq^{\st} S_2$, then the fact that $i \colon \Sub_{\mathbb{S}} \to \mathfrak{X}^*$ is a Heyting transformation implies that  $\mathfrak{X}^*(f)(i_A(S_1)) \leq^{\st} \mathfrak{X}^*(f)(i_A(S_2))$, in which case $\Sub^{\st}(f)(S_1) \leq^{\st} \Sub^{\st}(f)(S_2)$. We'll argue that it has a left adjoint defined by $\exists^{\st}_f(S) \coloneqq \{\exists^{\sigma}_{\iota(f)} (i_A(S)) \}^{\sigma}$ and a right adjoint given by $\forall^{\st}_f(S) \coloneqq \{ \forall^{\sigma}_{\iota(f)} (i_A(S))\}^{\sigma}$ for each $S \in \Sub^{\st}(A)$.
 
 \begin{lemma}
 Suppose topos-theoretic standardisation holds. Then for each map $f \colon A \to B$ in $\mathcal{S}$, we have an adjoint triple $\exists^{\st}_f \dashv \Sub^{\st}(f) \dashv \forall^{\st}_f$, where $\Sub^{\st}(f): \Sub^{\st}(B) \to \Sub^{\st}(A)$ is given by pulling back along $f$.
 \label{lem:StandardisationImpliesAdjoints}
 \end{lemma}

 \begin{proof}
  Functoriality of $\exists_f^{\st}$ and $\forall_f^{\st}$ follows from  $\exists^{\sigma}_{\iota(f)}$ and $\forall^{\sigma}_{\iota(f)}$ preserving the standard elements order and the definition of standardisation. Indeed, suppose $S_1 \leq^{\st} S_2$. Then $i_A(S_1)\leq^{\st} i_A(S_2)$, and since $\exists^{\sigma}_{\iota(f)}$ preserves the standard elements order we have $\exists^{\sigma}_{\iota(f)} (i_A(S_1)) \leq^{\st} \exists^{\sigma}_{\iota(f)} (i_A(S_2))$. But then 
  \begin{equation*}
      i_A(\{\exists^{\sigma}_{\iota(f)} (i_A(S_1)\}^{\sigma}) \cong^{\st} \exists^{\sigma}_{\iota(f)} (i_A(S_1)) \leq^{\st} \exists^{\sigma}_{\iota(f)} (i_A(S_2)) \cong^{\st} i_A(\{\exists^{\sigma}_{\iota(f)} (i_A(S_2))\}^{\sigma}),
  \end{equation*}
  and therefore $\exists^{\st}_f(S_1) \leq^{\st} \exists^{\st}_f(S_2)$. The proof of functoriality of $\forall_f^{\st}$ is analogous. \\
 To see that $\exists^{\st}_f \dashv \Sub^{\st}(f)$, suppose that $\exists^{\st}_f(S_1) \leq^{\st} S_2$. Then 
 \begin{equation*}
     \exists^{\sigma}_{\iota(f)} (i_A(S_1))  \cong^{\st} i_B(\exists^{\st}_f(S_1)) \leq^{\st} i_B(S_2),
 \end{equation*} by definition of $\exists^{\st}_f(S_1)$. Now $\exists^{\sigma}_{\iota(f)} \dashv^{\st} \mathfrak{X}^*(f)$ by Proposition~\ref{prop:SigmaAdjunction}, therefore $i_A(S_1) \leq^{\st} \mathfrak{X}^*(f)(i_B(S_2))= i_A(f^{-1}(S_2))$, as $i \colon \Sub_{\mathbb{S}} \to \mathfrak{X}^*$ is a Heyting transformation. Thus $S_1 \leq^{\st} \Sub^{\st}(f)(S_2)$. These steps are fully reversible, yielding the adjunction.\\
 The exact same argument shows that $\Sub^{\st}(f) \dashv \forall^{\st}_f$, as we know that $\mathfrak{X}^*(f) \dashv^{\st} \forall^{\sigma}_{\iota(f)}$ by Proposition~\ref{prop:SigmaAdjunction}.
 \end{proof}

 This as far as we can get with topos-theoretic standardisation alone, but we will show that if topos-theoretic transfer is also assumed then  $\Sub^{\st}$ can be upgraded to a hyperdoctrine. \\
 
 We'll need a useful consequence of classical Transfer: \emph{standard} objects that have the same standard elements must in fact be the same. This is a heavily utilised fact in Nelson-style arguments, which we can reproduce in the topos sphere:
 
 \begin{proposition}
 Given a Nelson structure satisfying topos-theoretic transfer and subobjects $S_1$ and $S_2$  of $A$ in $\mathbb{S}$, we have $S_1 \leq^{\st} S_2$ iff $i_A(S_1) \leq i_A(S_2)$.
 \label{prop.TransferImpliesContainment}
 \end{proposition}
 \begin{proof}
For the non-trivial direction, the hypothesis gives 
\begin{equation*}
    \sigma_A \leq i_A(S_1) \Rightarrow i_A(S_2) = i_A(S_1 \Rightarrow S_2)
\end{equation*}
in $\mathfrak{X}^*(A)$, while $\forall_!(\sigma_A \Rightarrow i_A(S_1 \Rightarrow S_2)) \land \sigma_1 \leq \forall_!(i_A(S_1 \Rightarrow S_2))$ by topos-theoretic transfer (applied to the standard map $!\colon A \to 1$). Combining these facts yields $\forall_!(\top_A) \land \sigma_1 =\top \land \sigma_1 =\top\leq \forall_!(i_A(S_1 \Rightarrow S_2))$. But then $i_A(S_1) \leq i_A(S_2)$, as $\sigma_1=\top$.
 \end{proof}

We can now prove:

\begin{lemma}
If a Nelson structure satisfies topos-theoretic standardisation and transfer, then $\Sub^{\st}$ is a hyperdoctrine over $\mathbb{S}$.
\label{lem.TransferImpliesHyperdoctrine}
\end{lemma}
\begin{proof}
By Lemma~\ref{lem:StandardisationImpliesAdjoints}, it remains to check that $\Sub^{\st}$ satisfies the Beck-Chevalley condition. Let \begin{equation*}
     \begin{tikzcd}
      A \arrow[swap]{d}{g} \arrow[]{r}{f}& B \arrow[]{d}{h}\\
      I \arrow[swap]{r}{k} &   J
    \end{tikzcd}
 \end{equation*}
 be a pullback in $\mathbb{S}$. We need to show that $\exists^{\st}_f\Sub^{\st}(g)(S) \cong^{\st} \Sub^{\st}(h) \exists^{\st}_k(S)$ for all $S \in \Sub^{\st}(I)$.
 
The left-hand side is $\exists^{\st}_f(g^{-1}(S)) = \{\exists^{\sigma}_{\iota(f)} i_A(g^{-1}(S))\}^{\sigma}$, so applying $i_B$ yields
 \begin{align*}
     i_B (\exists^{\st}_f\Sub^{\st}(g)(S)) \cong^{\st} \exists^{\sigma}_{\iota(f)} i_A(g^{-1}(S))= \exists^{\mathfrak{X}^*}_f (\sigma_A \land i_A(g^{-1}(S)))
 \end{align*}
 by definition of $\exists_{\iota(f)}^{\sigma}$. The fact that $i$ commutes with pullbacks then gives
 \begin{equation*}
     i_B (\exists^{\st}_f\Sub^{\st}(g)(S)) \cong^{\st}\exists^{\mathfrak{X}^*}_f (\sigma_A \land \mathfrak{X}^*(g)(i_I(S))).
 \end{equation*}
On the right-hand side, we have
\begin{equation*}
    {\Sub^{\st}(h) \exists^{\st}_k(S)}=h^{-1}(\{ \exists_{\iota(k)}^{\sigma}(i_I(S))\}^{\sigma})= \{\mathfrak{X}^*(h) \exists_{\iota(k)}^{\sigma}(i_I(S))\}^{\sigma}
\end{equation*}
by naturality of standardisation. Therefore
 \begin{equation*}
     i_B( \Sub^{\st}(h) \exists^{\st}_k(S)) \cong^{\st} \mathfrak{X}^*(h) \exists_{\iota(k)}^{\sigma}(i_I(S)).
 \end{equation*}
The problem is now reduced to showing that
 \begin{equation}
    \exists^{\mathfrak{X}^*}_f (\sigma_A \land \mathfrak{X}^*(g)(i_I(S)))\cong^{\st} \mathfrak{X}^*(h) \exists^{\mathfrak{X}^*}_k(\sigma_I \land i_I(S)).
     \label{eq:HaveToShow}
 \end{equation}
 We know  $\sigma_A \leq \mathfrak{X}^*(g)(\sigma_I)$, so $\exists^{\mathfrak{X}^*}_f (\sigma_A \land \mathfrak{X}^*(g)(i_I(S))) \leq \exists_f^{\mathfrak{X}^*} \mathfrak{X}^*(g)(\sigma_I \land i_I(S))$. Now $\iota$ preserves pullbacks and $\mathfrak{X}$ is a hyperdoctrine, hence $\mathfrak{X}^*$ satisfies the Beck-Chevalley condition. Applying it to the given pullback diagram gives $\exists^{\mathfrak{X}^*}_f (\sigma_A \land \mathfrak{X}^*(g)(i_I(S)))\leq \mathfrak{X}^*(h) \exists^{\mathfrak{X}^*}_k(\sigma_I \land i_I(S))$.
 
 For the converse inequality, note that the Beck-Chevalley property of $\mathfrak{X}^*$ gives  
     \begin{equation*}
         \mathfrak{X}^*(h) \exists^{\mathfrak{X}^*}_k(\sigma_I \land i_I(S))\leq \exists_f^{\mathfrak{X}^*}\mathfrak{X}^*(g)(\sigma_I \land i_I(S)),
     \end{equation*}
     which is bound by $\exists_f^{\mathfrak{X}^*}(i_A(g^{-1}S))$ by naturality of $i$. But topos-theoretic transfer gives $\exists_f^{\mathfrak{X}^*}(i_A(g^{-1}S)) \leq^{\st} \exists^{\sigma}_{\iota(f)}(i_A(g^{-1}S))$, and therefore
     \begin{equation*}
          \mathfrak{X}^*(h) \exists^{\mathfrak{X}^*}_k(\sigma_I \land i_I(S)) \leq^{\st} \exists^{\mathfrak{X}^*}_f(\sigma_A \land i_A(g^{-1}S)).
     \end{equation*}
     But $i_A(g^{-1}S) = \mathfrak{X}^*(g)(i_I(S))$.
\end{proof}

The natural question is how this newly minted doctrine embeds into $\mathfrak{X}^*$; this is in the expected way:

\begin{lemma}
For a Nelson structure satisfying topos-theoretic standardisation and transfer, the {$\mathbb{S}$-indexed} functor $i \colon \Sub^{\st} \to \mathfrak{X}^*$ is a Heyting transformation.
\label{lem.TransferImpliesHeyting}
\end{lemma}
\begin{proof}
 If $ \bowtie$ stands for any of the logical connectives, we clearly have that $(S_1  \bowtie^{\st} S_2) \cong^{\st} (S_1  \bowtie S_2)$ by the definition of the operations $\bowtie^{\st}$ in $\Sub^{\st}$, and thus  $i_A(S_1 \bowtie^{\st} S_2) \cong i_A(S_1  \bowtie S_2)$ by Proposition~\ref{prop.TransferImpliesContainment}. As $i_A$ preserves the operations in $\Sub_{\mathbb{S}}$, it therefore preserves the operations in $\Sub^{\st}$.
 
 It remains to show that it commutes with the quantifiers of $\Sub^{\st}$. We know $\exists^{\mathfrak{X}^*}_f i_A(S) \cong i_B(\exists^{\mathbb{S}}_f(S))$, so it suffices to show that $i_B(\exists^{\mathbb{S}}_f(S)) \cong i_B( \exists^{\st}_f(S))$. Again, this is the case by Proposition~\ref{prop.TransferImpliesContainment} since $\exists^{\mathbb{S}}_f(S) \cong^{\st} \exists^{\st}_f(S)$. The exact same argument applies to universal quantification.
\end{proof}

\begin{observation*}
If $\mathbb{S}$ is a topos, then $\Sub^{\st}$ admits a generic predicate: let $t \in \Sub_{\mathbb{S}}(\Omega)$ be a subobject classifier and $S \in \Sub^{\st}(A)$. If $f \colon A \to \Omega$ classifies $S$ in $\mathbb{S}$, then $\Sub^{\st}(t)=f^{-1}(t) \cong S$ in $\Sub_{\mathbb{S}}(A)$, and therefore $\Sub^{\st}(t) \cong^{\st} S$. In particular, if $\mathbb{S}$ is a topos and both topos-theoretic standardisation and transfer are assumed, then $\Sub^{\st}$ is a tripos over $\mathbb{S}$.
\end{observation*}

Note that a converse to the combination of Lemma~\ref{lem.TransferImpliesHeyting} and Lemma~\ref{lem.TransferImpliesHyperdoctrine} holds:

\begin{proposition}
Suppose topos-theoretic standardisation holds and the $\mathbb{S}$-indexed functor ${i \colon \Sub^{\st} \to \mathfrak{X}^*}$ preserves quantification. Then topos-theoretic transfer holds.
\label{prop:PreservationOfQuantifiersImpliesTransfer}
\end{proposition}
\begin{proof}
Under the given hypothesis, Lemma~\ref{lem:StandardisationImpliesAdjoints} still holds and we have $\exists_f(i_A(S)) \cong i_B( \exists^{\st}_f(S))$, and $\forall_f i_A(S) \cong i_B( \forall^{\st}_f(S))$ for each $S \in \Sub^{\st}(A)$. To prove existential topos-theoretic transfer we need $\exists_f(i_A(S))\cong^{\st} \exists_f^{\sigma}(i_A(S))$, therefore it suffices to show that $i_B \exists_f^{\st}(S) \cong^{\st} \exists_f^{\sigma}(i_A(S))$. But that is immediate from the definition of $\exists^{\st}_f$. The proof for universal quantification is dual.
\end{proof}


We will now work with such a Nelson structure for which topos-theoretic transfer and topos-theoretic standardisation hold, so that we have a hyperdoctrine $\Sub^{\st}$ with the previously discussed properties. It can be characterised as the doctrine of fixed points for a certain transformation on the hyperdoctrine of external predicates, mirroring the idea in classical IST that a subset of a standard set is itself standard iff it coincides with its own standardisation. 

We first observe that the transformation $i$ used to view standard predicates externally is \emph{conservative}, as its components are split monomorphisms:

\begin{proposition}
The $\mathbb{S}$-indexed functor $i \colon \Sub^{\st} \to \mathfrak{X}^*$ is a split monomorphism.
\label{prop:InclusionSplits}
\end{proposition}
\begin{proof}
The $\mathbb{S}$-indexed functor $(-)^{\sigma}\colon \mathfrak{X}^* \to \Sub^{\st}$ serves as a splitting for $i$ by uniqueness of standardisation.
\end{proof}

Moreover, this transformation factors through the doctrine of internal predicates, so that standard predicates are also internal:

\begin{proposition}
The Heyting transformation $i \colon \Sub^{\st} \to \mathfrak{X}^*$ of $\mathbb{S}$-indexed hyperdoctrines factors through $[\cdot]^* \colon \Sub_{\mathcal{I}}^* \to \mathfrak{X}^*$.
\label{prop:InclusionFactors}
\end{proposition}
\begin{proof}
By Proposition~\ref{prop.TransferImpliesContainment}, if $S_1 \leq^{\st} S_2$, then $[\iota(S_1)] \leq [\iota(S_2)]$, and thus $\iota(S_1) \leq \iota(S_2)$ since $[\cdot]$ is an embedding. This makes application of $\iota$ a functor $\Sub^{\st}(A) \to \Sub_{\mathcal{I}}^*(A)$ for each $A \in \mathbb{S}$. As $\Sub^{\st}(f)$ is just pulling back along $f$ and $\iota$ preserves pullbacks, such maps in fact form an $\mathbb{S}$-indexed functor $\Sub^{\st} \to \Sub_{\mathcal{I}}^*$.
\end{proof}

Given an external predicate $\phi \in \mathfrak{X}^*(A)$, the predicate $\diamond_A(\phi) \coloneqq i_A(\phi^{\sigma})$ is just the standardisation of $\phi$ (seen externally). It is clear from topos-theoretic standardisation that the familiar statement of Standardisation holds in the internal logic of $\mathfrak{X}^*$:
\begin{proposition}
Let $A$ be a standard object and $\phi \in \mathfrak{X}^*(A)$ be an external predicate in a Nelson structure satisfying topos-theoretic transfer and standardisation. Then 
    \begin{equation*}
       \mathfrak{X}^* \models  \forall x\colon A.\ \sigma_A(x) \Rightarrow  (x \in_A \phi \iff x \in \diamond_A(\phi))
    \end{equation*}
    \label{prop:standardisationIL}
\end{proposition}

The classical transfer statements also hold within the internal logic of $\mathfrak{X}^*$:
\begin{proposition}
Let $A$ be a standard object and $\psi \in \Sub(A)$ be a standard predicate in a Nelson structure satisfying topos-theoretic transfer and standardisation. Then 
    \begin{equation}
     \mathfrak{X}^* \models (   \exists x\colon A.\   \psi^*(x) ) \Rightarrow ( \exists x\colon A.\ \sigma_A(x) \land  \psi^*(x)),
     \label{eq:existentialtransferIL}
    \end{equation}
    and
    \begin{equation}
     \mathfrak{X}^* \models (   \forall x\colon A.\ \sigma_A(x) \Rightarrow  \psi^*(x) ) \Rightarrow (\forall x\colon A.\  \psi^*(x))
     \label{eq:universaltransferIL}
    \end{equation}
    where $\psi^* \coloneqq i_A(\psi) \in \mathfrak{X}^*(A)$.
    \label{prop:transferIL}
\end{proposition}

\begin{proof}
    To show (\ref{eq:existentialtransferIL}) is to check that
    \begin{equation*}
        \exists^{\mathfrak{X}^*}_{!_A}(i_A(\psi)) \leq \exists^{\mathfrak{X}^*}_{!_A}(\sigma_A \land i_A(\psi))
    \end{equation*}
    in $\mathfrak{X}^*(1)$, where $!_A\colon A \to 1$ in $\mathcal{E}$. Using that $i \colon \Sub_{\mathcal{E}} \to \mathfrak{X}^*$ is a Heyting transformation and topos-theoretic transfer, we can write 
    \begin{equation*}
       \exists^{\mathfrak{X}^*}_{!_A}(i_A(\psi)) \cong i_1(\exists_{!_A}^{\mathcal{E}}(\psi)) \leq^{\st} \exists^{\sigma}_{\iota(f)}(i_A(\psi)) = \exists^{\mathfrak{X}^*}_{!_A}(\sigma_A \land i_A(\psi)).
    \end{equation*}
    Now $\sigma_1=\top$ in $\mathfrak{X}^*(1)$, so we have the necessary inequality. The same argument applies to (\ref{eq:universaltransferIL}).
\end{proof}

The interpretation of intuitionistic first-order logic is sound in any first-order hyperdoctrine such as $\mathfrak{X}^*$, thus combining Proposition~\ref{prop:standardisationIL}, Proposition~\ref{prop:transferIL} and the statement of $K$-finite idealisation, we conclude:

\begin{theorem}[Soundess for Nelson models]
Let $\mathfrak{N}$ be a Nelson model with standard universe $\mathcal{E}$, $\mathcal{L}$ be an $\mathcal{E}$-typed relational language, and $\phi$ be a sentence in $\mathcal{L}$. If $\phi$ is provable by using intuitionistic logic without equality, $\Delta_0$-Transfer, Standardisation, and $K$-finite Idealisation, then $\mathfrak{X}^* \models \phi$ for any choice of interpretation of $\mathcal{L}$ in $\mathfrak{X}^*$.
\label{thm:SoundnessForNelsonModels}
\end{theorem}

This is typically used in conjunction with  Proposition~\ref{prop:InclusionFactors} and the fact that $i\colon \Sub^{\st} \to \mathfrak{X}^*$ is conservative.

\section{Ultrapower models}
\label{sec:UltrapowerModels}

This section discusses examples of Nelson models built from a given topos $\mathcal{E}$. This follows a similar approach to Robinson's nonstandard analysis, but employs ultrafilters and ultrapowers internal to $\mathcal{E}$ and a factorisation theorem due to A. Kock and C. J. Mikkelsen \cite{TTFNSE}.

\begin{definition}[Volger, 1974]
    Let $\mathcal{E}$ be a topos. An internal filter on an object $X$ of $\mathcal{E}$ is a subobject $U \rightarrowtail \Omega^X$ whose characteristic morphism $\chi_U$ preserves $\land \colon \Omega \times \Omega \to \Omega$. Such a filter is proper if $\chi_U \leq \exists (!_X)$, and is an ultrafilter if it is proper and $\chi_U$ preserves $\Rightarrow \colon \Omega \times \Omega \to \Omega$.
\end{definition}

Fix an internal ultrafilter $U$ on an object $X$ of $\mathcal{E}$, and let $A$ be an object of $\mathcal{E}$. The \emph{ultrapower} of $A$ modulo $U$ is the internal filtered colimit of the $A^K$ with $K \in U$. Alternatively, it can be built from a partial map representer for $A$ \cite{VOLGER1975345}, which we briefly recap:

\begin{definition}
    Let $A$ be an object of $\mathcal{E}$. A partial map representer for $A$ is an object $\Tilde{A}$ together with a monomorphism $\eta_A \colon A \rightarrowtail \Tilde{A}$ such that, for all partial maps \begin{tikzcd}
            D \arrow[]{r}{f} \arrow[rightarrowtail]{d}{} & A  \\
            B  &
        \end{tikzcd}
     there is an unique morphism $B \xrightarrow[]{\Tilde{f}} \Tilde{A}$ such that the diagram
        \begin{tikzcd}
            D \arrow[]{r}{f} \arrow[rightarrowtail]{d}{} & A \arrow[rightarrowtail]{d}{\eta_A}\\
            B \arrow[swap]{r}{\Tilde{f}}& \Tilde{A}
        \end{tikzcd}
    is a pullback.
\end{definition}
Such a partial map representer exists for any object $A$ in a topos \cite{Johnstone:592033}.\\

To build the ultrapower, first extract from the power object $\Tilde{A}^X$ the subobject $\Tilde{A}^X \slash U$ obtained by pulling back $U$ along $\chi_{\eta_A}^X$. This is the object of partial morphisms $X \to A$ with domain in the filter $U$.

Next, pull this subobject back along the composite
\begin{equation*}
    \Tilde{A}^X \slash U \times \Tilde{A}^X \slash U \rightarrowtail \Tilde{A}^X \times \Tilde{A}^X \xrightarrow[]{\Tilde{\cap}^X} \Tilde{A}^X
\end{equation*} to obtain $K_U(A) \xrightarrow{k_{U,A}} \Tilde{A}^X \slash U \times \Tilde{A}^X\slash U$; this object $K_U(A)$ is then the subobject of pairs $f,g$ of partial morphisms $X \to A$ with domain in $U$ such that $\textbf{Eq}(f_{\mid_{\dom(f) \cap \dom(g)}}, g_{\mid_{\dom(f) \cap \dom(g)}}) \in U$. 

Finally, the ultrapower $A^X \slash U$ is the coequaliser $\Tilde{A}\slash U \xrightarrow{q_{U,A}} A^X \slash U$ of the pair \begin{equation*}
    \pi_1 \circ k_{U,A}, \pi_2 \circ k_{U,A} \colon K_U(A) \to \Tilde{A}^X \slash U,
\end{equation*} where $\pi_1, \pi_2$ are the product projections. 

This construction is functorial in $A$. Moreover,  if the functor $\Tilde{(-)}^X$ preserves epimorphisms (which is the case when $\mathcal{E}$ satisfies the internal axiom of choice), then the ultrapower functor $(-)^X\slash U$ is a Heyting functor; details can be found in \cite{VOLGER1975345}.

In $\mathbf{Set}$, $A^X \slash U$ is the set of equivalence classes of partial maps $X \to A$ with domain in $U$, where $[f]=[g]$ when $\dom(f \cap g) \in U$. Classically, there is an embedding $A \to A^X \slash U$ given by taking $U$-almost everywhere constant maps, which is replicated in abstract by the existence of a natural transformation $d_U \colon \textbf{id}_{\mathcal{E}} \to (-)^X \slash U$, namely the composite 
\begin{equation*}
    A \xrightarrow{A^{!_X}} A^X \xrightarrow{\eta_{U,A}} \Tilde{A}^X \slash U \xrightarrow{q_{U,A}} A^X \slash U
\end{equation*}
where $\eta_{U,A}$ is the fruit of the universal property of the pullback defining $\Tilde{A}^X \slash U$ via $\eta_A^X$; this is monic since $U$ is proper, and $d_U(\Omega)$ is an isomorphism since $U$ is an ultrafilter \cite{VOLGER1975345}. We shall soon see that the monos $d_U(A) \colon A \rightarrowtail A^X \slash U$ serve as the predicates $\sigma_A$ of standard elements for a Nelson structure over $\mathcal{E}$.

The standard universe $\mathbb{S}$ for the aforementioned Nelson structure will, of course, be the topos $\mathcal{E}$. In order to construct an internal universe and a hyperdoctrine $\mathfrak{X}$, however, we will employ the following theorem, with $\phi = (-)^X \slash U$:

\begin{theorem}[A. Kock \& C. J. Mikkelsen, 1972]
Let $\phi \colon \mathcal{E} \to \mathcal{F}$ be a Heyting functor between toposes that preserves subobject classifiers. Then there is a topos $\mathcal{I}$ and a factorisation $\phi = \beta \circ \alpha$ such that $\alpha \colon \mathcal{E} \to \mathcal{I}$ is logical and $\beta\colon \mathcal{I} \to \mathcal{F}$ is a Heyting functor that preserves elements.
\label{thm:Kock-Mikkelsen}
\end{theorem}

Details of the construction of $\mathcal{I}$$, \alpha$, and $\beta$ can be found in \cite{TTFNSE} and \cite{TriposModels}. While it won't be relevant to this paper, we note that internal subobjects of $A^X\slash U$ can be identified with global elements $1 \to \wp(A)^X\slash U$.\\

We define the hyperdoctrine $\mathfrak{X}$ of external predicates to be the hyperdoctrine $\beta^*(\Sub_{\mathcal{F}})$ over $\mathcal{I}$. Upon reindexing along $\alpha$, we obtain the hyperdoctrine $\mathfrak{X}^* \coloneqq \alpha^* (\beta^* (\Sub_{\mathcal{F}}))) \cong \phi^* \Sub_{\mathcal{F}} = \Sub_{\mathcal{F}}(-^X\slash U)$ over $\mathcal{E}$. We can then define $[\cdot] \colon \Sub_{\mathcal{I}} \to \mathfrak{X}$ via the action of $\beta$ on subobjects, and $\iota \colon \Sub_{\mathcal{E}} \to \Sub_{\mathcal{I}}$ through the action of the (faithful) logical functor $\alpha$ on subobjects.

The final piece is determining the external standardness predicate. For each $A \in \mathcal{E}$, define $\sigma_A$ to be the monomorphism $d_U(A)\colon A \rightarrowtail A^X \slash U$ in $\mathfrak{X}^*(A)$. To show lax naturality of $\sigma$, we need $d_U(A) \leq (f^X \slash U)^{-1}(d_U(B))$ for all standard morphisms $f\colon A \to B$ (i.e., morphisms in $\mathcal{E}$). For that, consider the pullback of $d_U(B)$ along $f^X \slash U$:
\begin{equation*}
    \begin{tikzcd}
       (f^X \slash U)^{-1}(d_U(B)) \arrow[rightarrowtail]{d}{}  \arrow[]{r}{}& B \arrow[rightarrowtail]{d}{}\\
        A^X \slash U  \arrow[swap]{r}{f^X\slash U}& B^X \slash U
    \end{tikzcd}
\end{equation*}

The naturality diagram of $d_U$ with respect to $f$ fits inside this diagram, and therefore there is an unique morphism $A \to (f^X \slash U)^{-1}(d_U(B)) $ witnessing the inequality $\sigma_A \leq \mathfrak{X}^*(f) (\sigma_B)$.

\begin{definition}
    Let $U$ be an ultrafilter on an object $X$ of a topos that satisfies the internal axiom of choice. A Nelson structure induced by $U$ (denoted $\mathfrak{N}_U$) is one derived from a Kock-Mikkelsen factorisation of the Heyting functor $(-)^X \slash U$, as described above.
\end{definition}

We will now show that every Nelson structure of the form $\mathfrak{N}_U$ satisfies topos-theoretic transfer and standardisation. Recall that for such structures the Heyting transformation $i \colon \Sub_{\mathcal{E}} \to \mathfrak{X}^*$ maps $S \rightarrowtail A$ to $S^X\slash U \rightarrowtail A^X \slash U$. 

For the remainder of this chapter, we will drop the subscripts in $\Sub_{\mathcal{E}}(A)$.

\begin{lemma}
    Let $S \rightarrowtail A$ be a standard subobject in $\mathfrak{N}_U$. Then the pullback projection $\sigma_A \cap S^X\slash U \rightarrowtail A$ is isomorphic to $S$ in $\Sub(A)$.
    \label{lemma:key}
\end{lemma}
\begin{proof}
    The naturality square of $d_U$ with respect to the standard map $S \rightarrowtail A$ is a cone for the diagram over which the pullback $\sigma_A \land S^X \slash U$ is defined, thus we get an inequality $S \leq \sigma_A \cap S^X \slash U$ in $\Sub(A)$. 
    The other direction will be obtained via the internal language of $\mathcal{E}$ by building an appropriate total functional relation. We can internally describe
    \begin{equation*}
        \sigma_A \land S^X\slash U = \{a=[f\colon K \to S \rightarrowtail A] \in A^X \slash U \colon(\exists \Tilde{a}).\ a=[c_{\Tilde{a}}]\},
    \end{equation*}
    where $K \in U$ and $c_{\Tilde{a}} \colon X \to A$ is the constant map $x \mapsto \Tilde{a}$. Internally, we thus have
    \begin{equation*}
         \sigma_A \land S^X\slash U = \{a \in A^X \slash U \colon (\exists s \in S).\ a=[c_s]\}.
    \end{equation*}
Since two constant partial maps that coincide on some domain must be equal, it is clear that the internal relation associating  $[c_s] \in \sigma_A \land S^X\slash U$ to $s \in S$ is total and functional. As toposes are regular categories, this implies the existence of a morphism witnessing $\sigma_A \land S^X \slash U \leq S$ in $\Sub(A)$. 
\end{proof}

\begin{corollary}
    Nelson structures of the form $\mathfrak{N}_U$ satisfy topos-theoretic standardisation.
\end{corollary}

\begin{proof}
    Given an external predicate $W \rightarrowtail A^X \slash U$, define $W^\sigma$ to be the pullback projection $\sigma_A \cap W \rightarrowtail A$ in $\Sub(A)$. 
    By Lemma~\ref{lemma:key}, we have
    \begin{equation}
        \sigma_A \land i_A(W^{\sigma}) = \sigma_A \land (W^{\sigma})^X\slash U \cong (W^{\sigma} \rightarrowtail A \xrightarrow{\sigma_A} A^X \slash U) = \sigma_A \cap W
        \label{eq:stand}
    \end{equation}
    in $\Sub(A^X \slash U)$. As $(\sigma_A \cap W) \cap \sigma_A = \sigma_A \cap W \leq W$, we have $i_A(W^{\sigma}) \leq^{\st} W \cap \sigma_A \leq^{\st} W$.
    Conversely, Equation~\ref{eq:stand} gives $\sigma_A \cap W \cong \sigma_A \land i_A(W^{\sigma}) \leq i_A(W^{\sigma})$, i.e., $W \leq^{\st} i_A(W^{\sigma})$.
    This settles the existence of the standardisation of external predicates. 
    
    Uniqueness is due to $\sigma_A$ being monic: if $Q \in \Sub(A)$ is such that $i_A(Q) \cong^{\st} W \cong^{\st} i_A(W^{\sigma})$, then by the same argument in Equation~\ref{eq:stand} we have
    \begin{equation*}
        \sigma_A \land i_A(Q) \cong \sigma_A \circ Q \leq i_A(W^{\sigma}) \land \sigma_A = \sigma_A \circ W^{\sigma},
    \end{equation*}
    whence $Q \leq W^{\sigma}$. By symmetry, $W^{\sigma} \leq Q$ as well.
\end{proof}

\begin{lemma}
    Nelson structures of the form $\mathfrak{N}_U$ satisfy topos-theoretic transfer.
\end{lemma}

\begin{proof}
   Note that $\exists^{\sigma}_{\iota(f)}(S)$ in this Nelson structure is $\exists^{\mathcal{F}}_{f^X \slash U}(\sigma_A \land S^X\slash U)$  for each standard $S \in \Sub(A)$, by definition of $\mathfrak{X}$. 
    Using the fact that $i_B$ is a Heyting transformation, existential transfer reduces to showing that
    \begin{equation*}
       \exists^{\mathcal{F}}_{f^X \slash U}(S^X \slash U) \cong^{\st} \exists^{\mathcal{F}}_{f^X \slash U}(\sigma_A \land S^X\slash U)
    \end{equation*}
   in $\mathfrak{X}^*(B)=\Sub_{\mathcal{F}}(B^X \slash U)$ for each standard $f \colon A \to B$ and $S \rightarrowtail A$. Moreover, every topos that satisfies the internal axiom of choice is Boolean (performing the argument in \cite{ChoiceImpliesEMGoodmanMyhill} or internalising that of \cite{ChoiceImpliesEMDiaconescu}), therefore universal transfer follows from the existential one.

   For the easy direction, note that $\sigma_A \land S^X \slash U \leq (f^X\slash U)^{-1}(\sigma_B) \land S^X \slash U$ by lax naturality of $\sigma$, thus
   \begin{equation*}
       \sigma_B \land \exists_{f^X\slash U}(\sigma_A \land S^X\slash U) \leq \exists_{f^X\slash U} (S^X \slash U \land (f^X \slash U)^{-1}(\sigma_B)) = \sigma_B \land \exists_{f^X\slash U}(S^X \slash U)
   \end{equation*}
   by Frobenius reciprocity of existential quantification. But that is bound by $\exists_{f^X\slash U}(S^X \slash U)$, so that 
   \begin{equation*}
       \exists_{f^X\slash U}(\sigma_A \land S^X\slash U)\leq^{\st} \exists_{f^X\slash U}(S^X \slash U).
   \end{equation*}

   For the converse, we write $\sigma_B \land \exists_{f^X\slash U}(S^X \slash U) = \sigma_B \circ \exists_f(S)$ and, similarly, $\exists_{f^X\slash U}(\sigma_A \land S^X\slash U) = \exists_{f^X \slash U}(\sigma_A \circ S)$ in $\Sub(B^X \slash U)$ using Lemma~\ref{lemma:key}, so that is suffices to show that
   $\sigma_B \circ \exists_f(S) \leq \exists_{f^X \slash U} (\sigma_A \circ S)$ in $\Sub(B^X\slash U)$. We do so by using the characterisation of existential quantification in a topos in terms of images.

   First we get a mono $\gamma\colon \exists_{f^X \slash U}(\sigma_A \circ S) = \im(f^X \slash U \circ \sigma_A \circ S) \rightarrowtail B$ from the universal property of the image, since naturality of $d_U$ makes $\sigma_B \circ (f \circ S)$ a factorisation of $f^X \slash U \circ \sigma_A \circ S$ through $\sigma_B \colon B \rightarrowtail B^X \slash U$. It makes the diagram
   \begin{equation*}
       \begin{tikzcd}
        & &\exists_{f^X \slash U}(\sigma_A \circ S) \arrow[rightarrowtail]{dd}{\gamma}\arrow[rightarrowtail]{dr}{} & \\
        S\arrow[]{urr}{} \arrow[rightarrowtail]{dr}{}& & & B^X\slash U   \\
        & A\arrow[]{r}{f} &B\arrow[rightarrowtail, swap]{ur}{\sigma_B}& 
       \end{tikzcd}
   \end{equation*}
   commute. Now, the square on the left is a factorisation of $f \circ S$ through  the monomorphism $\gamma$, therefore the universal property of $\exists_f(S) = \im(f \circ S)$ provides a a monomorphism $\delta\colon \exists_f(S) \to \exists_{f^X\slash U}(\sigma_A \circ S)$ making the diagram 
\begin{equation*}
    \begin{tikzcd}
        S \arrow[]{r}{}\arrow[]{dr}{} & \exists_f(S)\arrow[rightarrowtail]{r}{}\arrow[rightarrowtail]{d}{\delta} & B\\
         & \exists_{f^X\slash U}(\sigma_A \circ S) \arrow[rightarrowtail, swap]{ur}{\gamma}&
    \end{tikzcd}
\end{equation*}
commute. But then $\delta$ witnesses $\sigma_B \circ \exists_f(S) \leq \exists_{f^X\slash U}(\sigma_A \circ S)$ in $\Sub(B^X \slash U)$, as needed.
\end{proof}

It remains to address bounded idealisation. We will show that, if $\mathcal{E}$ is a topos satisfying the internal axiom of choice, then for each object $B$ of $\mathcal{E}$ there is a Nelson structure of the form $\mathfrak{N}_U$ that satisfies bounded idealisation relative to $B$-valued internal relations.

Our argument is adapted from the proof attributed to William C. Powell in \cite{nelson1977}; it resorts to the fact that any topos satisfying the internal axiom of choice is Boolean and internally believes in certain core results, namely the fact that any family of subsets that satisfies the finite intersection property is contained in some filter, the ultrafilter principle, and the axiom of choice itself. 

The proofs that follow will be carried out in the internal language of the topos $\mathcal{E}$. We start with the following:

\begin{lemma}
    Let $\mathcal{E}$ be a topos satisfying the internal axiom of choice. For each $B \in \mathcal{E}$, there is an object $X$ and an ultrafilter $U$ on $X$ such that, for every ultrafilter $V$ on $B$, there is $\xi \in B^X\slash U$ with
    $V = \{E \subseteq B\colon \xi \in E^X \slash U\}$ in $\Sub(\wp(B))$.
    \label{lemma:adequate}
\end{lemma}

\begin{proof}
    Let $X=K(\wp(B))$ be the object of finite subobjects of $\wp(B)$, and internally define $E^!\coloneqq \{x \in X \colon E \in x\}$ for each $E \in \Sub(B)$. 
    A finite intersection of sets of the form $E^!$ is inhabited (e.g., $\cap_{j=1}^n {E_j}^!$ is inhabited by $\{E_1, \cdots, E_n\} \in X$)\footnote{Recall that all the typical notions of finiteness coincide within a Boolean topos.}, thus the family of all the $E^!$ with $E \subseteq B$ satisfies the finite intersection property. They therefore generate a filter, which can be extended to an ultrafilter $U$ by the ultrafilter principle.
    
    Let $V$ be an ultrafilter on $B$ and, for each $x \in X$, internally define $B_x \coloneqq \bigcap \{E \in x \colon E \in V\}$. Each $B_x$ is in $V$ (as filters are closed under finite intersections), and therefore $B_x$ is inhabited (as $V$ is an ultrafilter).
    Using the internal axiom of choice, there is $\xi_x \in B_x$ for each $x \in X$. Define $\xi$ as the equivalence class of the partial map $x \mapsto \xi_x$ modulo $U$.
    
    Let $E \in V$. Note that if $x \in X$ and $E \in x$, then $B_x \subseteq E$. As $\xi_x \in B_x$, we conclude that $\xi_x \in E$. In summary, $E^! \subseteq \{x \in X \colon \xi_x \in E\}$ for all $E \in V$. But $U$ contains all the $E^!$ and is a terminal segment of $X$, thus $\{x \in X \colon \xi_x \in E\} \in U$ for all $E \in V$. We conclude that $\xi \in E^X\slash U$ for all $E \in V$, i.e., $V \subseteq \{E \subseteq B\colon \xi \in E^X \slash U\}$. As ultrafilters are maximal, this is actually an equality.
\end{proof}

An ultrapower $B^X\slash U$ with respect to an ultrafilter satisfying the property in the conclusion of Lemma~\ref{lemma:adequate} is said to be an \emph{adequate ultrapower}.

\begin{corollary}
     Let $A$ and $B$ be objects of $\mathcal{E}$, $B^X\slash U$ be an adequate ultrapower of $B$, and $R \colon 1 \to \wp(A \times B)^X \slash U$ be an internal relation. Then $\mathfrak{N}_U$ satisfies the $R$-instance of the axiom of bounded idealisation.
\end{corollary}

\begin{proof}
   Suppose it is internally true that for every finite $F \in \Sub(A)$ there is $y \in B$ such that $R([c_x],[c_y])$ holds for all $x \in F$. For each $F \in K(A)$, the subobjects
    \begin{equation*}
        \{b \in B \colon (\forall a \in F) R( [c_a], [c_b])\} \rightarrowtail B
    \end{equation*}
    are thus inhabited. There is therefore a filter on $B$ generated by them, which can be extended to an ultrafilter $V$ on $B$.
    
    As $B^X\slash U$ is adequate, we have $V=\{E \in \wp(B) \colon \xi \in E^X\slash U\}$ for some element $\xi \in B^X \slash U$. Taking $F=\{a\}$, we note that $V$ contains all the subobjects for the form $\{b \in B \colon R([c_a],[c_b])\}$ for $a \in A$. So ${\xi \in \{b \in B \colon R([c_a],[c_b])\}^X\slash U}$ for each $a \in A$, i.e., $R([c_a], \xi)$ holds.
\end{proof}

Combining the results of this section, we obtain:

\begin{theorem}
    Let $\mathcal{E}$ be a topos satisfying the internal axiom of choice and $B$ be an object of $\mathcal{E}$. Then there is an ultrafilter $U$ on an object of $\mathcal{E}$ such that $\mathfrak{N}_U$ satisfies topos-theoretic standardisation, topos-theoretic transfer, and every instance of $B$-bounded standardisation.
\end{theorem}

In summary, topos-theoretic transfer and standardisation hold for any Nelson structure of the form $\mathfrak{N}_U$ whose standard universe satisfies the internal axiom of choice, but they do not necessarily satisfy bounded idealisation. However, we can always find an ultrafilter $U$ for which $\mathfrak{N}_U$ also satisfies $B$-bounded idealisation, for any given $B$. 

There are situations where the limitations on idealisation can be circumvented. If $\mathcal{E}$ is also locally small and complete (e.g., a Boolean \'etendue \cite{CategoriesAllegories} or a model for ZFA set theory), then one can construct a cumulative hiearchy within the topos \cite{FOURMAN198091}, and use the fact that every object occurs as a subobject of a sufficiently high level of the hierarchy. The same problem exists for classical Internal Set Theory, and this workaround was implemented in Theorem 8.7 of \cite{nelson1977}. 

 \bibliographystyle{unsrt} 
 \bibliography{main}





\end{document}